\documentclass[oneside]{amsart}
\usepackage[T1]{fontenc}
\UseRawInputEncoding
\usepackage{url}
\usepackage{amsbsy}
\usepackage{amstext}
\usepackage{amsthm}
\usepackage{amssymb}
\usepackage{stmaryrd}
\usepackage{geometry}
\makeatletter
\numberwithin{equation}{section}
\numberwithin{figure}{section}
\theoremstyle{plain}
\newtheorem{thm}{\protect\theoremname}
\theoremstyle{plain}
\newtheorem{conjecture}[thm]{\protect\conjecturename}
\theoremstyle{definition}
\newtheorem{defn}[thm]{\protect\definitionname}
\theoremstyle{plain}
\newtheorem{lem}[thm]{\protect\lemmaname}
\theoremstyle{plain}
\newtheorem{cor}[thm]{\protect\corollaryname}
\theoremstyle{remark}
\newtheorem{rem}[thm]{\protect\remarkname}
\theoremstyle{plain}
\newtheorem{prop}[thm]{\protect\propositionname}

\usepackage{listings}
\usepackage{shuffle}

\makeatother

\providecommand{\conjecturename}{Conjecture}
\providecommand{\corollaryname}{Corollary}
\providecommand{\definitionname}{Definition}
\providecommand{\lemmaname}{Lemma}
\providecommand{\propositionname}{Proposition}
\providecommand{\remarkname}{Remark}
\providecommand{\theoremname}{Theorem}

\begin{document}
\address[Minoru Hirose]{Graduate School of Science and Engineering, Kagoshima University, 1-21-35 Korimoto, Kagoshima, Kagoshima 890-0065, Japan}
\email{hirose@sci.kagoshima-u.ac.jp}
\title{Conjectural duality for iterated $q$-integrals on $\mathbb{P}^{1}$
minus four generic points}
\author{Minoru Hirose}
\begin{abstract}
We propose a conjectural $q$-analogue of the classical duality for
iterated integrals on $\mathbb{P}^{1}$ minus four points, arising
from the involutive M\"{o}bius transformation which exchanges the
four marked points in pairs. To this end, we introduce iterated $q$-integrals
with position-dependent $q$-shifts of the parameters and define a
functional on admissible words in the six pairwise letters. The conjecture
states that this functional is invariant under a natural anti-automorphism
of the word algebra. We relate the conjecture to Yamamoto's duality
for one-variable multiple $q$-polylogarithms. Finally, we prove the
conjecture in several special cases. 
\end{abstract}

\maketitle

\section{Introduction}

The main purpose of this paper is to formulate a conjectural $q$-analogue
of a classical duality for iterated integrals on the four-punctured
projective line, arising from an involutive M\"{o}bius transformation
that exchanges the four marked points in pairs. After briefly recalling
the background and the classical case, we state the conjecture.

\subsection{Background}

Iterated integrals of logarithmic differential forms on punctured
projective lines play an important role in several areas of mathematics.
From Chen's viewpoint \cite{Chen}, they provide a non-abelian refinement
of de Rham theory and encode the unipotent fundamental group \cite{Hain}.
Such iterated integrals give rise to hyperlogarithms and multiple
polylogarithms \cite{Goncharov_MPL}, and their special values include
multiple zeta values. They are also naturally related to mixed Hodge--Tate
structures, mixed Tate motives \cite{Goncharov_MPL}, and the geometry
of the moduli spaces $M_{0,n}$ \cite{BrownModuli}.

To describe the duality considered in this paper, it is convenient
to use the logarithmic differential forms 
\[
\omega_{a,b}:=\left(\frac{1}{t-a}-\frac{1}{t-b}\right)dt\qquad(a,b\in\mathbb{P}^{1}(\mathbb{C})).
\]
For any M\"{o}bius transformation $\varphi$, one has 
\[
\varphi^{*}\omega_{a,b}=\omega_{\varphi^{-1}(a),\,\varphi^{-1}(b)},
\]
and hence $\varphi$ induces relations among the corresponding iterated
integrals. In particular, for
\[
\mathcal{D}=\{A,B,C,D\}\subset\mathbb{P}^{1}(\mathbb{C}),
\]
there is an involutive M\"{o}bius transformation exchanging $A$ with
$D$ and $B$ with $C$. The resulting duality gives a natural identity
for iterated integrals on $\mathbb{P}^{1}(\mathbb{C})\setminus\mathcal{D}$.
In particular, this includes the classical duality relation for multiple
zeta values as a special case. The conjecture considered in this paper
is a $q$-analogue of this duality.

\subsection{\label{subsec:notation}Notation}

To state the classical duality and its $q$-analogue in a parallel
form, we introduce the following notation. Let $W$ be the set of
finite words in the six letters $e_{AB},\ e_{AC},\ e_{AD},\ e_{BC},\ e_{BD},\ e_{CD}.$
Although the symbols $A,B,C,D$ will also be used as parameters, the
letters $e_{AB},\dots,e_{CD}$ are regarded as purely formal symbols.
Thus, for example, $e_{AB}$ and $e_{AC}$ are distinct even if $B=C$.
When there is no risk of confusion, we simply write $AB,\dots,CD$,
or $(AB),\dots,(CD)$, in place of $e_{AB},\dots,e_{CD}$. Define
an anti-automorphism $\tau:W\to W$ by
\[
\tau(e_{uv}):=e_{\sigma(v)\sigma(u)}\qquad(\sigma(A)=D,\ \sigma(B)=C,\ \sigma(C)=B,\ \sigma(D)=A).
\]
We say that $w\in W$ is admissible if it does not begin with $e_{Av}$
$(v=B,C,D)$ and does not end with $e_{uD}$ $(u=A,B,C)$. We denote
by $W^{0}$ the set of admissible words in $W$.

\subsection{\label{subsec:C_case}The classical case}

For $x,y,u_{1},\dots,u_{k},v_{1},\dots,v_{k}\in\mathbb{R}$ satisfying
\begin{itemize}
\item $x\leq y$,
\item for every $t$ with $x<t<y$, one has $t\notin\{u_{1},\dots,u_{k},v_{1},\dots,v_{k}\}$,
\item $x\notin\{u_{1},v_{1}\}$ and $y\notin\{u_{k},v_{k}\}$,
\end{itemize}
we define the following variant of Goncharov's iterated integral symbol
\cite{Goncharov_MPL}:
\[
I(x;[u_{1},v_{1}],\dots,[u_{k},v_{k}];y):=\int_{x<t_{1}<\cdots<t_{k}<y}\prod_{j=1}^{k}\left(\frac{1}{t_{j}-u_{j}}-\frac{1}{t_{j}-v_{j}}\right)dt_{j}.
\]

The following theorem gives the classical case of our conjecture,
and we include it for completeness.
\begin{thm}
\label{thm:Duality_for_real_case}Let $A,B,C,D$ be distinct real
numbers such that $A<D$ and $B,C\notin[A,D]$. Define $L:W^{0}\to\mathbb{R}$
by
\[
L((u_{1}v_{1})\cdots(u_{k}v_{k})):=I(A;[u_{1},v_{1}],\dots,[u_{k},v_{k}];D).
\]
Then, for every $w\in W^{0}$, we have
\[
L(w)=L(\tau(w)).
\]
\end{thm}

\begin{proof}
Let $\varphi$ be a M\"{o}bius transformation which sends $(A,B,C,D)$
to $(D,C,B,A)$. For $w=e_{u_{1}v_{1}}\cdots e_{u_{k}v_{k}}$, the
change of variables $t_{j}\mapsto\varphi(t_{j})$ transforms the integral
defining $L(w)$ into 
\[
(-1)^{k}\int_{D>t_{1}>\cdots>t_{k}>A}\prod_{j=1}^{k}\left(\frac{1}{t_{j}-\varphi^{-1}(u_{j})}-\frac{1}{t_{j}-\varphi^{-1}(v_{j})}\right)dt_{j}.
\]
Reversing the order of the variables by $t_{j}\mapsto t_{k+1-j}$,
we obtain $L(\tau(w))$.
\end{proof}

\subsection{\label{subsec:q_case}The $q$-case and the conjecture}

Let $K$ be a field and let $q\in K^{\times}$. For simplicity, we
assume that $q^{N}\neq1$ for every $N\in\mathbb{Z}_{\geq1}$. We
define a partial order on $K^{\times}$ by 
\[
x\trianglelefteq y\Longleftrightarrow y/x\in\{q^{-n}\mid n\in\mathbb{Z}_{\geq0}\}.
\]
Suppose that $x,y\in K^{\times}$ and $u_{1},\dots,u_{k},v_{1},\dots,v_{k}\in K$
satisfy the following conditions: 
\begin{itemize}
\item $x\trianglelefteq y$, 
\item for every $t$ such that $x\trianglelefteq t\trianglelefteq y$, we
have $t\notin\{u_{1},\dots,u_{k},v_{1},\dots,v_{k}\}$. 
\end{itemize}
Then we define the corresponding iterated $q$-integral symbol by
\[
I_{q}(x;[u_{1},v_{1}],\dots,[u_{k},v_{k}];y):=\sum_{x\trianglelefteq t_{1}\trianglelefteq\cdots\trianglelefteq t_{k}\trianglelefteq y}\prod_{j=1}^{k}\left(\frac{t_{j}}{t_{j}-u_{j}}-\frac{t_{j}}{t_{j}-v_{j}}\right)\in K.
\]
This symbol has the following classical limit: for $0<x<y$, setting
$q_{N}=(x/y)^{1/N}$, one has 
\[
I(x;[u_{1},v_{1}],\dots,[u_{k},v_{k}];y)=\lim_{N\to\infty}(1-q_{N})^{k}I_{q_{N}}(x;[u_{1},v_{1}],\dots,[u_{k},v_{k}];y).
\]

The following is the main conjecture of this paper.
\begin{conjecture}
\label{conj:main}Let $N\geq0$, and let $K$ be the field of fractions
of $\mathbb{Q}[A,B,C,D,q]/(A-q^{N}D)$. Define $L_{q}:W^{0}\to K$
by 
\[
L_{q}((u_{1}v_{1})\cdots(u_{k}v_{k})):=I_{q}(A;[u_{1}^{(1)},v_{1}^{(1)}],\dots,[u_{k}^{(k)},v_{k}^{(k)}];D).
\]
Here, for $u_{j},v_{j}\in\{A,B,C,D\}$, the quantities $u_{j}^{(j)}$
and $v_{j}^{(j)}$ are obtained by multiplying $u_{j}$ and $v_{j}$
by suitable powers of $q$, explicitly as follows: 
\begin{align*}
A^{(j)} & :=Aq^{\#\{h\leq j\,\mid\,u_{h}v_{h}\in\{BC,BD,CD\}\}},\\
B^{(j)} & :=Bq^{\#\{h\leq j\,\mid\,u_{h}v_{h}\notin\{AC,AD\}\}+\#\{h\geq j\,\mid\,u_{h}v_{h}=CD\}},\\
C^{(j)} & :=Cq^{\#\{h\leq j\,\mid\,u_{h}v_{h}\notin\{AB,AD\}\}+\#\{h\geq j\,\mid\,u_{h}v_{h}=BD\}},\\
D^{(j)} & :=Dq^{-\#\{h\geq j\,\mid\,u_{h}v_{h}\in\{AB,AC,BC\}\}}.
\end{align*}
Then, for any $w\in W^{0}$, we have
\[
L_{q}(w)=L_{q}(\tau(w)).
\]
\end{conjecture}

The $q$-shifts appearing in the conjecture were suggested by numerical
experiments. The SageMath code used for computational checks of the
conjecture is available at \cite{duality_code}.

In later sections, we will consider various specializations of the
parameters $A,B,C,D$. Although such substitutions are not defined
for arbitrary elements of $K$, the expressions $L_{q}(w)$ arising
from the above definition admit these specializations in all cases
considered in this paper. We use the notation $F_{A=a,\,B=b,\,C=c,\,D=d}$
to denote the simultaneous specialization of the parameters $A,B,C,D$
to $a,b,c,d$, respectively; when only some variables are displayed,
only those variables are specialized.

The remainder of this paper is organized as follows. In Section~2,
we collect preliminary material. In Section~3, we explain how Yamamoto's
duality for one-variable multiple $q$-polylogarithms can be viewed
as a special case of Conjecture~\ref{conj:main}. In Section~4,
we prove Conjecture~\ref{conj:main} in several special cases.

\subsection*{Acknowledgements}

This work was supported by JSPS KAKENHI Grant Number JP22K03244.

\section{Preliminaries}
\begin{defn}
For positive integers $k_{1},\dots,k_{d}$, we define the multiple
$q$-polylogarithms, following Zhao \cite{Zhao}, by 
\begin{align*}
\mathrm{Li}_{q;k_{1},\dots,k_{d}}(z_{1},\dots,z_{d}) & :=\sum_{0<n_{1}<\cdots<n_{d}}\frac{z_{1}^{n_{1}}\cdots z_{d}^{n_{d}}}{(1-q^{n_{1}})^{k_{1}}\cdots(1-q^{n_{d}})^{k_{d}}}\\
 & =\sum_{0<n_{1}<\cdots<n_{d}}\frac{x_{1}^{n_{1}}x_{2}^{n_{2}-n_{1}}\cdots x_{d}^{n_{d}-n_{d-1}}}{(1-q^{n_{1}})^{k_{1}}\cdots(1-q^{n_{d}})^{k_{d}}}\in\mathbb{Q}[[q,x_{1},\dots,x_{d}]],
\end{align*}
where we set $x_{j}:=\prod_{i=j}^{d}z_{i}$. 
\end{defn}

\begin{defn}
For $x,y$ with $x\trianglelefteq y$, we define 
\[
I_{q}(x;u_{1},\dots,u_{k};y):=\sum_{x\trianglelefteq t_{1}\trianglelefteq\cdots\trianglelefteq t_{k}\trianglelefteq y}\prod_{j=1}^{k}\frac{t_{j}}{t_{j}-u_{j}}.
\]
We extend this definition to the case $x=0$ by setting 
\begin{align*}
I_{q}(0;u_{1},\dots,u_{k};y) & :=\lim_{N\to\infty}I_{q}(yq^{N};u_{1},\dots,u_{k};y)\\
 & =\sum_{t_{1}\trianglelefteq\cdots\trianglelefteq t_{k}\trianglelefteq y}\prod_{j=1}^{k}\frac{t_{j}}{t_{j}-u_{j}}.
\end{align*}
\end{defn}

It is known that multiple $q$-polylogarithms admit the following
$q$-integral expression.
\begin{thm}[{\cite[Corollary 7.4]{Zhao}}]
\label{thm:series_expression}We have
\[
\mathrm{Li}_{q;k_{1},\dots,k_{d}}(z_{1},\dots,z_{d})=(-1)^{d}I_{q}(0;a_{1},\{0\}^{k_{1}-1},\dots,a_{d},\{0\}^{k_{d}-1};1)
\]
where we set
\[
a_{i}:=\frac{1}{z_{i}\cdots z_{d}}\qquad(1\leq i\leq d).
\]
\end{thm}

\begin{proof}
By \cite[Corollary 7.4]{Zhao}, the statement follows from a direct
comparison with our definition, up to normalization.
\end{proof}
The iterated $q$-integrals satisfy the following $q$-difference
formula.
\begin{lem}[$q$-difference formula]
\label{lem:difference_formula}Let $0\leq h\leq k$. Assume $a_{h}\neq a_{h+1}$.
Then, we have
\begin{align*}
 & I_{q}(a_{0};a_{1},\dots,a_{k};a_{k+1})-I_{q}(a_{0};a_{1},\dots,a_{h},a_{h+1}q,\dots,a_{k}q;a_{k+1}q)\\
 & =\begin{cases}
\frac{a_{h}I_{q}(a_{0};a_{1},\dots,\widehat{a_{h+1}},\dots,a_{k};a_{k+1})}{a_{h}-a_{h+1}} & h<k\\
0 & h=k
\end{cases}\\
 & \quad+\begin{cases}
\frac{a_{h+1}I_{q}(a_{0};a_{1},\dots,\widehat{a_{h}},\dots,a_{k};a_{k+1})}{a_{h+1}-a_{h}} & h>0\\
0 & h=0.
\end{cases}
\end{align*}
\end{lem}

\begin{proof}
We first consider the case $0<h<k$. Then
\begin{align*}
 & I_{q}(a_{0};a_{1},\dots,a_{h},a_{h+1}q,\dots,a_{k}q;a_{k+1}q)\\
 & =\sum_{a_{0}\trianglelefteqslant t_{1}\trianglelefteqslant\cdots\trianglelefteqslant t_{k}\trianglelefteqslant a_{k+1}q}\prod_{j=1}^{h}\frac{t_{j}}{t_{j}-a_{j}}\prod_{j=h+1}^{k}\frac{t_{j}}{t_{j}-a_{j}q}\\
 & =\sum_{a_{0}\trianglelefteqslant t_{1}\trianglelefteqslant\cdots\trianglelefteqslant t_{h}\vartriangleleft u_{h+1}\trianglelefteqslant\cdots\trianglelefteqslant u_{k}\trianglelefteqslant a_{k+1}}\prod_{j=1}^{h}\frac{t_{j}}{t_{j}-a_{j}}\prod_{j=h+1}^{k}\frac{u_{j}}{u_{j}-a_{j}}\qquad(u_{j}:=t_{j}q^{-1})
\end{align*}
where $x\vartriangleleft y$ means $x\trianglelefteqslant y$ and
$x\neq y$. Thus,
\begin{align*}
 & I_{q}(a_{0};a_{1},\dots,a_{k};a_{k+1})-I_{q}(a_{0};a_{1},\dots,a_{h},a_{h+1}q,\dots,a_{k}q;a_{k+1}q)\\
 & =\sum_{a_{0}\trianglelefteqslant t_{1}\trianglelefteqslant\cdots\trianglelefteqslant t_{h}=t_{h+1}\trianglelefteqslant\cdots\trianglelefteqslant t_{k}\trianglelefteqslant a_{k+1}}\prod_{j=1}^{k}\frac{t_{j}}{t_{j}-a_{j}}\\
 & =\sum_{a_{0}\trianglelefteqslant t_{1}\trianglelefteqslant\cdots\trianglelefteqslant t_{h}=t_{h+1}\trianglelefteqslant\cdots\trianglelefteqslant t_{k}\trianglelefteqslant a_{k+1}}\frac{t_{h}}{t_{h}-a_{h}}\frac{t_{h}}{t_{h}-a_{h+1}}\prod_{j\neq h,h+1}\frac{t_{j}}{t_{j}-a_{j}}\\
 & =\frac{1}{a_{h}-a_{h+1}}\sum_{a_{0}\trianglelefteqslant t_{1}\trianglelefteqslant\cdots\trianglelefteqslant t_{h}=t_{h+1}\trianglelefteqslant\cdots\trianglelefteqslant t_{k}\trianglelefteqslant a_{k+1}}\left(\frac{a_{h}t_{h}}{t_{h}-a_{h}}-\frac{a_{h+1}t_{h}}{t_{h}-a_{h+1}}\right)\prod_{j\neq h,h+1}\frac{t_{j}}{t_{j}-a_{j}}\\
 & =\frac{a_{h}I_{q}(a_{0};a_{1},\dots,\widehat{a_{h+1}},\dots,a_{k};a_{k+1})-a_{h+1}I_{q}(a_{0};a_{1},\dots,\widehat{a_{h}},\dots,a_{k};a_{k+1})}{a_{h}-a_{h+1}}.
\end{align*}
Thus, the case $0<h<k$ is proved. The cases $h=0$ and $h=k$ follow
from similar but simpler arguments.
\end{proof}
\begin{cor}
\label{cor:difference_formula_0}We have
\begin{align*}
 & I_{q}(a_{0};a_{1},\dots,a_{h},0,a_{h+1},\dots,a_{k};a_{k+1})-I_{q}(a_{0};a_{1},\dots,a_{h},0,a_{h+1}q,\dots,a_{k}q;a_{k+1}q)\\
 & =I_{q}(a_{0};a_{1},\dots,a_{h},a_{h+1},\dots,a_{k};a_{k+1}).
\end{align*}
\end{cor}

\section{\label{sec:known_cases}The connection to Yamamoto's duality for
one-variable multiple $q$-polylogarithms}

In \cite{Yam_duality}, Yamamoto proved a duality for a $q$-analogue
of one-variable multiple polylogarithms arising from iterated integrals
on $\mathbb{P}^{1}\setminus\{0,1,\infty,z\}$. As will be explained
later, this result provides a common generalization of the duality
formulas for $q$-MZVs in the Bradley--Zhao and Schlesinger--Zudilin
models \cite{Bradley_qMZVs}\cite{Zhao_sz_duality}.

In this section, we show that Yamamoto's duality can be viewed as
a special case of Conjecture~\ref{conj:main}. We first introduce
the notation following \cite{Yam_duality}.

A pair $\tilde{k}=(k,\mu)\in\mathbb{Z}_{>0}\times\{0,1\}$ is called
an augmented positive integer, and a tuple $\tilde{\boldsymbol{k}}=(\tilde{k}_{1},\dots,\tilde{k}_{r})$
of augmented positive integers is called an augmented index. We say
that $\tilde{\boldsymbol{k}}$ is admissible if either $r=0$ or $\tilde{k}_{r}\neq(1,1)$.
For an augmented index 
\[
\tilde{\boldsymbol{k}}=\bigl((k_{1},\mu_{1}),\dots,(k_{r},\mu_{r})\bigr),
\]
we define 
\[
w(\tilde{\boldsymbol{k}}):=y_{\mu_{1}}x^{k_{1}-1}\cdots y_{\mu_{r}}x^{k_{r}-1}\in\mathbb{Q}\langle x,y_{0},y_{1}\rangle.
\]
Here the letters $x,y_{0},y_{1}$ are related to the differential
forms in \cite{Yam_duality} by 
\[
(x,y_{0},y_{1})=(e_{0},-e_{z},e_{z}-e_{1}).
\]
Let $\tau'$ be the anti-automorphism of $\mathbb{Q}\langle x,y_{0},y_{1}\rangle$
defined by 
\[
\tau'(x):=y_{1},\qquad\tau'(y_{0}):=y_{0},\qquad\tau'(y_{1}):=x.
\]
For an admissible augmented index $\tilde{\boldsymbol{k}}$, its dual
$\tilde{\boldsymbol{k}}^{\dagger}$ is defined by the relation 
\[
\tau'\bigl(w(\tilde{\boldsymbol{k}})\bigr)=w(\tilde{\boldsymbol{k}}^{\dagger}).
\]
For an augmented index 
\[
\tilde{\boldsymbol{k}}=\bigl((k_{1},\mu_{1}),\dots,(k_{r},\mu_{r})\bigr),
\]
we define 
\[
\mathrm{Li}_{q}^{(1)}(\tilde{\boldsymbol{k}};z):=\sum_{0=m_{0}<m_{1}<\cdots<m_{r}}\prod_{i=1}^{r}\frac{q^{(k_{i}-1)m_{i}}\bigl(\mu_{i}+(-1)^{\mu_{i}}q^{m_{i-1}}z^{m_{i}-m_{i-1}}\bigr)}{(1-q^{m_{i}})^{k_{i}}}\in\mathbb{Q}[[z,q]].
\]
Yamamoto proved the following duality theorem.
\begin{thm}[Yamamoto \cite{Yam_duality}]
\label{thm:Yam_duality}For an admissible augmented index $\tilde{\boldsymbol{k}}$,
we have $\mathrm{Li}_{q}^{(1)}(\tilde{\boldsymbol{k}};z)=\mathrm{Li}_{q}^{(1)}(\tilde{\boldsymbol{k}}^{\dagger};z)$.
\end{thm}

\begin{rem}
This theorem can be regarded as a common generalization of the duality
formulas for $q$-MZVs in the Bradley--Zhao and Schlesinger--Zudilin
models. Let 
\[
\tilde{\boldsymbol{k}}=\bigl((k_{1},\mu_{1}),\dots,(k_{r},\mu_{r})\bigr)
\]
be an admissible augmented index such that $\mu_{1}=\cdots=\mu_{r}=1$.
Then its dual $\tilde{\boldsymbol{k}}^{\dagger}$ also has all $\mu$-components
equal to $1$. In this case, the specializations at $z=0$ and $z=q$
coincide with the $q$-MZVs in the Bradley--Zhao and Schlesinger--Zudilin
models, respectively: 
\[
\mathrm{Li}_{q}^{(1)}(\tilde{\boldsymbol{k}};0)=\zeta_{q}^{\mathrm{BZ}}(k_{1},\dots,k_{r}):=\sum_{0<m_{1}<\cdots<m_{r}}\prod_{i=1}^{r}\frac{q^{(k_{i}-1)m_{i}}}{(1-q^{m_{i}})^{k_{i}}},
\]
\[
\mathrm{Li}_{q}^{(1)}(\tilde{\boldsymbol{k}};q)=\zeta_{q}^{\mathrm{SZ}}(k_{1}-1,\dots,k_{r}-1):=\sum_{0<m_{1}<\cdots<m_{r}}\prod_{i=1}^{r}\left(\frac{q^{m_{i}}}{1-q^{m_{i}}}\right)^{k_{i}-1}.
\]
It follows that Theorem~\ref{thm:Yam_duality} generalizes both the
duality formulas for $q$-MZVs in the Bradley--Zhao model \cite[Corollary~3]{Bradley_qMZVs}
and Schlesinger--Zudilin model \cite[Theorem 8.3]{Zhao_sz_duality},
namely, 
\[
\zeta_{q}^{\mathrm{BZ}}(\{1\}^{l_{1}-1},k_{1}+1,\dots,\{1\}^{l_{r}-1},k_{r}+1)=\zeta_{q}^{\mathrm{BZ}}(\{1\}^{k_{r}-1},l_{r}+1,\dots,\{1\}^{k_{1}-1},l_{1}+1)\qquad(k_{j},l_{j}\geq1)
\]
\[
\zeta_{q}^{\mathrm{SZ}}(\{0\}^{l_{1}-1},k_{1},\dots,\{0\}^{l_{r}-1},k_{r})=\zeta_{q}^{\mathrm{SZ}}(\{0\}^{k_{r}-1},l_{r},\dots,\{0\}^{k_{1}-1},l_{1})\qquad(k_{j},l_{j}\geq1).
\]
\end{rem}

We extend the definition of $L_{q}$ to $\mathbb{Q}W^{0}$ by $\mathbb{Q}$-linearity.
Let $\theta:\mathbb{Q}\langle x,y_{0},y_{1}\rangle\to\mathbb{Q}W$
be an algebra homomorphism defined by $\theta(x):=e_{AC}-e_{BC}$,
$\theta(y_{0}):=e_{BC}$, $\theta(y_{1}):=e_{BD}-e_{BC}$. Then, by
direct calculation, we have the following:
\begin{prop}
\label{prop:Lq_Yam}For an augmented index $\tilde{\boldsymbol{k}}=\bigl((k_{1},\mu_{1}),\dots,(k_{r},\mu_{r})\bigr)$,
we have
\[
L_{q}(\theta(w(\tilde{\boldsymbol{k}})))_{A=0,B=\infty,C=z^{-1}q^{-k},D=1}=\mathrm{Li}_{q}^{(1)}(\tilde{\boldsymbol{k}};z)
\]
where $k=k_{1}+\cdots+k_{r}$.
\end{prop}

\begin{proof}
Write 
\[
w(\tilde{\boldsymbol{k}})=u_{1}\cdots u_{k}\in\mathbb{Q}\langle x,y_{0},y_{1}\rangle
\]
and put
\[
J:=\{j\in\{1,\dots,k\}\,\mid\,u_{j}\neq x\}.
\]
Then,
\begin{align*}
 & L_{q}(\theta(w(\tilde{\boldsymbol{k}})))_{A=0,B=\infty,C=z^{-1}q^{-k},D=1}\\
 & =(-1)^{\#\{i\,\mid\,u_{i}=y_{0}\}}\sum_{\substack{(\epsilon_{j})_{j\in J}\\
0\leq\epsilon_{j}\leq0\,(u_{j}=y_{0})\\
0\leq\epsilon_{j}\leq1\,(u_{j}=y_{1})
}
}\sum_{t_{1}\trianglelefteq\cdots\trianglelefteq t_{k}\trianglelefteq1}\prod_{j=1}^{k}\begin{cases}
1 & u_{j}=x\\
\frac{t_{j}}{t_{j}-C^{(j)}} & u_{j}\neq x,\epsilon_{j}=0\\
\frac{-t_{j}}{t_{j}-D^{(j)}} & u_{j}\neq x,\epsilon_{j}=1
\end{cases}
\end{align*}
where
\begin{align*}
C^{(j)} & :=z^{-1}q^{-k+j+\#\{h\geq j\,\mid\,u_{h}=y_{1},\epsilon_{h}=1\}},\\
D^{(j)} & :=q^{-1-k+j+\#\{h\geq j\,\mid\,u_{h}=y_{1},\epsilon_{h}=1\}}.
\end{align*}
Thus, we have
\begin{align*}
 & L_{q}(\theta(w(\tilde{\boldsymbol{k}})))_{A=0,B=\infty,C=z^{-1}q^{-k},D=1}\\
 & =(-1)^{\#\{i\,\mid\mu_{i}=0\}}\sum_{\substack{\boldsymbol{\epsilon}=(\epsilon_{1},\dots,\epsilon_{r})\in\{0,1\}^{r}\\
0\leq\epsilon_{i}\leq\mu_{i}
}
}(-1)^{\epsilon_{1}+\cdots+\epsilon_{r}}I_{q}(0;a_{1}^{\boldsymbol{\epsilon}},\{0\}^{k_{1}-1},\dots,a_{r}^{\boldsymbol{\epsilon}},\{0\}^{k_{r}-1};1)
\end{align*}
where $a_{i}^{\boldsymbol{\epsilon}}:=b_{i}^{\boldsymbol{\epsilon}}(z^{-1}q)^{1-\epsilon_{i}}$
with
\[
b_{i}^{\boldsymbol{\epsilon}}:=q^{-\sum_{j=i}^{r}(k_{j}-1+\delta_{\epsilon_{j},0})}.
\]
We also put $b_{r+1}^{\boldsymbol{\epsilon}}:=1$. Then, by Theorem
\ref{thm:series_expression}, 
\begin{align*}
 & L_{q}(\theta(w(\tilde{\boldsymbol{k}})))_{A=0,B=\infty,C=z^{-1}q^{-k},D=1}\\
 & =\sum_{0=m_{0}<m_{1}<\cdots<m_{r}}\sum_{\substack{\boldsymbol{\epsilon}=(\epsilon_{1},\dots,\epsilon_{r})\in\{0,1\}^{r}\\
0\leq\epsilon_{i}\leq\mu_{i}
}
}(-1)^{(\mu_{1}-\epsilon_{1})+\cdots+(\mu_{r}-\epsilon_{r})}\prod_{i=1}^{r}\frac{(a_{i}^{\boldsymbol{\epsilon}})^{m_{i-1}-m_{i}}}{(1-q^{m_{i}})^{k_{i}}}\\
 & =\sum_{0=m_{0}<m_{1}<\cdots<m_{r}}\sum_{\substack{\boldsymbol{\epsilon}=(\epsilon_{1},\dots,\epsilon_{r})\in\{0,1\}^{r}\\
0\leq\epsilon_{i}\leq\mu_{i}
}
}(-1)^{(\mu_{1}-\epsilon_{1})+\cdots+(\mu_{r}-\epsilon_{r})}\prod_{i=1}^{r}\frac{(b_{i+1}^{\boldsymbol{\epsilon}}/b_{i}^{\boldsymbol{\epsilon}})^{m_{i}}}{(1-q^{m_{i}})^{k_{i}}}\prod_{\substack{1\leq i\leq r\\
\epsilon_{i}=0
}
}(zq^{-1})^{m_{i}-m_{i-1}}\\
 & =\sum_{0=m_{0}<m_{1}<\cdots<m_{r}}\sum_{\substack{\boldsymbol{\epsilon}=(\epsilon_{1},\dots,\epsilon_{r})\in\{0,1\}^{r}\\
0\leq\epsilon_{i}\leq\mu_{i}
}
}\prod_{i=1}^{r}\frac{q^{(k_{i}-1)m_{i}}}{(1-q^{m_{i}})^{k_{i}}}\prod_{\substack{1\leq i\leq r\\
\epsilon_{i}=0
}
}(-1)^{\mu_{i}}q^{m_{i-1}}z^{m_{i}-m_{i-1}}\\
 & =\sum_{0=m_{0}<m_{1}<\cdots<m_{r}}\prod_{i=1}^{r}\frac{q^{(k_{i}-1)m_{i}}\left(\mu_{i}+(-1)^{\mu_{i}}q^{m_{i-1}}z^{m_{i}-m_{i-1}}\right)}{(1-q^{m_{i}})^{k_{i}}}\\
 & =\mathrm{Li}_{q}^{(1)}(\tilde{\boldsymbol{k}};z).
\end{align*}
\end{proof}
Let $\tilde{\boldsymbol{k}}$ be an admissible augmented index, and
put 
\[
w:=\theta(w(\tilde{\boldsymbol{k}})).
\]
Since 
\[
\tau(w)=\theta(w(\tilde{\boldsymbol{k}}^{\dagger})),
\]
the specialization of Conjecture~\ref{conj:main} given by 
\[
L_{q}(w)_{A=0,\,B=\infty,\,C=z^{-1}q^{-k},\,D=1}=L_{q}(\tau(w))_{A=0,\,B=\infty,\,C=z^{-1}q^{-k},\,D=1}
\]
can be interpreted, via Proposition~\ref{prop:Lq_Yam}, as the equality
\[
\mathrm{Li}_{q}^{(1)}(\tilde{\boldsymbol{k}};z)=\mathrm{Li}_{q}^{(1)}(\tilde{\boldsymbol{k}}^{\dagger};z),
\]
which is precisely Theorem~\ref{thm:Yam_duality}.

\section{Proofs for some specific cases}

In this section, we prove Conjecture~\ref{conj:main} in several
special cases.

\subsection{The case $AD=BCq^{m(w)}$}

For a word $w=e_{u_{1}v_{1}}\cdots e_{u_{k}v_{k}}$, set 
\[
m(w):=1+\#\{h\mid u_{h}v_{h}\notin\{AD\}\}.
\]
In this subsection, we prove the duality relation 
\[
L_{q}(w)=L_{q}(\tau(w))
\]
under the specialization 
\[
AD=BCq^{m(w)}.
\]

\begin{lem}
\label{lem:inv_trans}We have
\[
I_{q}(x;[u_{1},v_{1}],\dots,[u_{k},v_{k}];y)=I_{q}(x;[\frac{xy}{v_{k}},\frac{xy}{u_{k}}],\dots,[\frac{xy}{v_{1}},\frac{xy}{u_{1}}];y).
\]
\end{lem}

\begin{proof}
By definition, we have
\[
I_{q}(x;[u_{1},v_{1}],\dots,[u_{k},v_{k}];y):=\sum_{x\trianglelefteq t_{1}\trianglelefteq\cdots\trianglelefteq t_{k}\trianglelefteq y}\prod_{j=1}^{k}\left(\frac{t_{j}}{t_{j}-u_{j}}-\frac{t_{j}}{t_{j}-v_{j}}\right).
\]
Substituting $t_{k+1-j}':=\frac{xy}{t_{j}}$, we obtain
\begin{align*}
 & =\sum_{x\trianglelefteq t_{1}'\trianglelefteq\cdots\trianglelefteq t_{k}'\trianglelefteq y}\prod_{j=1}^{k}\left(\frac{\frac{xy}{t_{j}'}}{\frac{xy}{t_{j}'}-u_{k+1-j}}-\frac{\frac{xy}{t_{j}'}}{\frac{xy}{t_{j}'}-v_{k+1-j}}\right)\\
 & =\sum_{x\trianglelefteq t_{1}'\trianglelefteq\cdots\trianglelefteq t_{k}'\trianglelefteq y}\prod_{j=1}^{k}\left(\frac{t_{j}'}{t_{j}'-\frac{xy}{v_{k+1-j}}}-\frac{t_{j}'}{t_{j}'-\frac{xy}{u_{k+1-j}}}\right)\\
 & =I_{q}(x;[\frac{xy}{v_{k}},\frac{xy}{u_{k}}],\dots,[\frac{xy}{v_{1}},\frac{xy}{u_{1}}];y).
\end{align*}
This proves the lemma.
\end{proof}
\begin{thm}
Let $w:=e_{u_{1}v_{1}}\cdots e_{u_{k}v_{k}}\in W^{0}$, and set
\[
m(w):=1+\#\{h\,\mid\,u_{h}v_{h}\notin\{AD\}\}.
\]
If $AD=BCq^{m(w)}$, then
\[
L_{q}(w)=L_{q}(\tau(w)).
\]
\end{thm}

\begin{proof}
By definition, the left-hand side is equal to 
\[
L_{q}(w)=I_{q}(A;[u_{1}^{(1)},v_{1}^{(1)}],\dots,[u_{k}^{(k)},v_{k}^{(k)}];D)
\]
where
\begin{align*}
A^{(j)} & :=Aq^{\#\{h\leq j\,\mid\,u_{h}v_{h}\in\{BC,BD,CD\}\}}\\
B^{(j)} & :=Bq^{\#\{h\leq j\,\mid\,u_{h}v_{h}\notin\{AC,AD\}\}+\#\{h\geq j\,\mid\,u_{h}v_{h}=CD\}}\\
C^{(j)} & :=Cq^{\#\{h\leq j\,\mid\,u_{h}v_{h}\notin\{AB,AD\}\}+\#\{h\geq j\,\mid\,u_{h}v_{h}=BD\}}\\
D^{(j)} & :=Dq^{-\#\{h\geq j\,\mid\,u_{h}v_{h}\in\{AB,AC,BC\}\}},
\end{align*}
and the right-hand side is equal to
\begin{equation}
L_{q}(\tau(w))=I_{q}(A;[\tau(v_{k})^{[k]},\tau(u_{k})^{[k]}],\dots,[\tau(v_{1})^{[1]},\tau(u_{1})^{[1]}];D)\label{eq:L_tau}
\end{equation}
where
\begin{align*}
A^{[j]} & :=Aq^{\#\{h\geq j\,\mid\,u_{h}v_{h}\in\{AB,AC,BC\}\}}\\
B^{[j]} & :=Bq^{\#\{h\geq j\,\mid\,u_{h}v_{h}\notin\{AD,BD\}\}+\#\{h\leq j\,\mid\,u_{h}v_{h}=AB\}}\\
C^{[j]} & :=Cq^{\#\{h\geq j\,\mid\,u_{h}v_{h}\notin\{AD,CD\}\}+\#\{h\leq j\,\mid\,u_{h}v_{h}=AC\}}\\
D^{[j]} & :=Dq^{-\#\{h\leq j\,\mid\,u_{h}v_{h}\in\{BC,BD,CD\}\}}.
\end{align*}
By applying Lemma~\ref{lem:inv_trans} to (\ref{eq:L_tau}), we have
\[
L_{q}(\tau(w))=I_{q}(A;[\frac{AD}{\tau(u_{1})^{[1]}},\frac{AD}{\tau(v_{1})^{[1]}}],\dots,[\frac{AD}{\tau(u_{k})^{[k]}},\frac{AD}{\tau(v_{k})^{[k]}}];D).
\]
Since
\[
\frac{AD}{\tau(u_{j})^{[j]}}=u_{j}^{(j)},\qquad\frac{AD}{\tau(v_{j})^{[j]}}=v_{j}^{(j)},
\]
the desired identity follows.
\end{proof}

\subsection{The case where $B=C=\infty$ and $w=e_{BD}^{k}e_{AB}^{l}$}

In this subsection, we prove the duality conjecture for $w=e_{BD}^{k}e_{AB}^{l}$
with $k,l\geq1$ under the specialization $B=C=\infty$. We first
reformulate the conjecture for a general word $w$ under this specialization
in an equivalent form. Assume $B=C=\infty$. By applying the change
of variables $t_{j}\mapsto AD/t_{k+1-j}$ to the variables of integration,
we obtain, for $w=u_{1}\cdots u_{k}\in W^{0}$, 
\[
L_{q}(\tau(w))_{B=C=\infty}=(-1)^{k}L_{q^{-1}}(w)_{B=C=\infty}.
\]
Thus the conjectural identity can be formulated in the form $L_{q}(w)_{B=C=\infty}=(-1)^{k}L_{q^{-1}}(w)_{B=C=\infty}$.
Moreover, it is no longer necessary to distinguish between $e_{AB}$
and $e_{AC}$, nor between $e_{BD}$ and $e_{CD}$; furthermore, if
a word contains $e_{BC}$, then its value under $L_{q}$ is zero.
Therefore, if we identify $e_{AB}$ and $e_{AC}$ with $x$, identify
$e_{BD}$ and $e_{CD}$ with $y$, and identify $e_{AD}$ with $z$,
the conjecture can be stated as follows.
\begin{conjecture}
\label{conj:B_C_inf}Let $\mathfrak{h}:=\mathbb{Q}\langle x,y,z\rangle$
and set $\mathfrak{h}^{0}:=\mathbb{Q}\oplus y\mathfrak{h}x$. Fix
$N\geq0$. Define
\[
Z_{N,q}:\mathfrak{h}^{0}\to\mathbb{Q}(q)
\]
by
\begin{align*}
Z_{N,q}(u_{1}\cdots u_{k}) & :=\sum_{0\leq n_{1}\leq\cdots\leq n_{k}\leq N}\prod_{\substack{1\leq j\leq k\\
u_{j}=x
}
}\frac{1}{1-q^{n_{j}+\#\{h\leq j\,\mid\,u_{h}=y\}}}\prod_{\substack{1\leq j\leq k\\
u_{j}=y
}
}\frac{-1}{1-q^{n_{j}-N-\#\{h\geq j\,\mid\,u_{h}=x\}}}\\
 & \qquad\qquad\times\prod_{\substack{1\leq j\leq k\\
u_{j}=z
}
}\left(\frac{1}{1-q^{n_{j}+\#\{h\leq j\,\mid\,u_{h}=y\}}}-\frac{1}{1-q^{n_{j}-N-\#\{h\geq j\,\mid\,u_{h}=x\}}}\right).
\end{align*}
Then, for any word $w=u_{1}\cdots u_{k}\in\mathfrak{h}^{0}$, 
\[
Z_{N,q^{-1}}(w)=(-1)^{k}Z_{N,q}(w).
\]
\end{conjecture}

In this subsection, we prove the following result.
\begin{thm}
Conjecture~\ref{conj:B_C_inf} holds for $w=y^{k}x^{l}$ with $k,l\geq1$.
\end{thm}

\begin{proof}
Let $\mathbb{Q}(q)_{\mathrm{ev}}$ denote the subspace on which the
involution $q\mapsto q^{-1}$ acts by multiplication by $+1$, and
let $\mathbb{Q}(q)_{\mathrm{od}}$ denote the subspace on which it
acts by multiplication by $-1$. For $k\in\mathbb{Z}$, set 
\[
\mathbb{Q}(q)_{k}:=\begin{cases}
\mathbb{Q}(q)_{\mathrm{ev}}, & \text{if }k\text{ is even},\\
\mathbb{Q}(q)_{\mathrm{od}}, & \text{if }k\text{ is odd}.
\end{cases}
\]
Then it suffices to prove that 
\[
f_{k,l}(N)\in\mathbb{Q}(q)_{k+l}
\]
where we put
\[
f_{k,l}(N):=I_{q}(1;\{q^{-N-l}\}^{k},\{q^{k}\}^{l};q^{-N})
\]
for $k,l\geq1$. We prove the assertion by induction on $k+l$. First,
we note that
\begin{equation}
f_{k,l}(N)=(-1)^{k+l}\left.f_{l,k}(N)\right|_{q\mapsto q^{-1}}.\label{eq:f_kl_change}
\end{equation}
The case $(k,l)=(1,1)$ follows directly from (\ref{eq:f_kl_change}).
Thus, it remains to consider the case $k+l>2$. Since the assertions
for $(k,l)$ and $(l,k)$ are equivalent by (\ref{eq:f_kl_change}),
we may assume without loss of generality that $l\geq2$. Set 
\[
g_{k,l,m,n}(N):=I_{q}(1;\{q^{-N-m}\}^{k},\{q^{n}\}^{l};q^{-N})\qquad(k,l\geq0,\ m,n\geq1)
\]
and $R_{n}:=1/(1-q^{n})$. Then, we have
\begin{align}
g_{k,l,m,n}(N)-g_{k,l,m,n+1}(N-1) & =R_{-m-N}g_{k-1,l,m,n}(N)\qquad(k,m,n\geq1,l\geq0),\label{eq:diff_01}\\
g_{k,l,m,n}(N)-g_{k,l,m+1,n+1}(N-1) & =R_{-m-n-N}g_{k-1,l,m,n}(N)+R_{m+n+N}g_{k,l-1,m,n}(N)\qquad(k,l,m,n\geq1),\label{eq:diff_11}\\
g_{k,l,m,n}(N)-g_{k,l,m+1,n}(N-1) & =R_{n+N}g_{k,l-1,m,n}(N)\qquad(k\geq0,l,m,n\geq1)\label{eq:diff_10}
\end{align}
by the $q$-difference formula (Lemma \ref{lem:difference_formula}).
Then,
\begin{align*}
 & f_{k,l}(N-1)\\
 & =g_{k,l,l-1,k}(N)-R_{k+N}f_{k,l-1}(N)\qquad(\text{by (\ref{eq:diff_10})})\\
 & =g_{k,l,l,k+1}(N-1)+R_{1-k-l-N}g_{k-1,l,l-1,k}(N)+\left(R_{k+l-1+N}-R_{k+N}\right)f_{k,l-1}(N)\qquad(\text{by (\ref{eq:diff_11})})\\
 & =f_{k,l}(N)-R_{-l-N}g_{k-1,l,l,k}(N)+R_{-k-l+1-N}g_{k-1,l,l-1,k}(N)+\left(R_{k+l-1+N}-R_{k+N}\right)f_{k,l-1}(N)\qquad(\text{by (\ref{eq:diff_01})}).
\end{align*}
We use this identity in both cases $k=1$ and $k>1$. If $k=1$, since
$g_{0,l,m,n}$ does not depend on $m$, this implies
\begin{align*}
f_{1,l}(N-1)-f_{1,l}(N) & =\left(R_{l+N}-R_{1+N}\right)f_{1,l-1}(N).
\end{align*}
Here, the right-hand side is an element of $\mathbb{Q}(q)_{1+l}$
since $f_{1,l-1}(N)\in\mathbb{Q}(q)_{l}$ by the induction hypothesis
and $R_{l+N}-R_{1+N}\in\mathbb{Q}(q)_{\mathrm{od}}$ by definition.
Note that 
\[
f_{k,l}(0)=\frac{1}{(1-q^{-l})^{k}(1-q^{k})^{l}}\in\mathbb{Q}(q)_{k+l}
\]
 for any $k,l\geq1$. Therefore, the claim for this case follows as
\[
f_{1,l}(N)=-\sum_{N'=1}^{N}\left(f_{1,l}(N'-1)-f_{1,l}(N')\right)+f_{1,l}(0)\in\mathbb{Q}(q)_{1+l}.
\]
Let $k\geq2$. By (\ref{eq:diff_11}), we have
\begin{align}
g_{k-1,l,l,k}(N) & =g_{k-1,l,l-1,k-1}(N+1)-R_{-k-l+1-N}g_{k-2,l,l-1,k-1}(N+1)-R_{k+l-1+N}f_{k-1,l-1}(N+1),\label{eq:g_eq1}
\end{align}
 and, by (\ref{eq:diff_01}), 
\begin{equation}
g_{k-1,l,l-1,k}(N)=g_{k-1,l,l-1,k-1}(N+1)-R_{-l-N}g_{k-2,l,l-1,k-1}(N+1).\label{eq:g_eq2}
\end{equation}
Hence,
\begin{align*}
 & f_{k,l}(N-1)-f_{k,l}(N)-\left(R_{k+l-1+N}-R_{k+N}\right)f_{k,l-1}(N)\\
 & =-R_{-l-N}g_{k-1,l,l,k}(N)+R_{-k-l+1-N}g_{k-1,l,l-1,k}(N)\\
 & =(R_{-k-l+1-N}-R_{-l-N})g_{k-1,l,l-1,k-1}(N+1)+R_{-l-N}R_{k+l-1+N}f_{k-1,l-1}(N+1)\qquad(\text{by (\ref{eq:g_eq1}) and (\ref{eq:g_eq2})})\\
 & =(R_{-k-l+1-N}-R_{-l-N})f_{k-1,l}(N)+(R_{-k-l+1-N}-R_{-l-N})R_{k+N}f_{k-1,l-1}(N+1)\\
 & \quad+R_{-l-N}R_{k+l-1+N}f_{k-1,l-1}(N+1)\qquad(\text{by (\ref{eq:diff_10})})\\
 & =(R_{-k-l+1-N}-R_{-l-N})f_{k-1,l}(N)\\
 & \quad+\left((R_{-k-l+1-N}-R_{-l-N})R_{k+N}+R_{-l-N}R_{k+l-1+N}\right)f_{k-1,l-1}(N+1).
\end{align*}
Using the identity $R_{m}=1-R_{-m}$ for $m\in\mathbb{Z}\setminus\{0\}$,
we obtain
\begin{align*}
 & (R_{-k-l+1-N}-R_{-l-N})R_{k+N}+R_{-l-N}R_{k+l-1+N}\\
 & =(R_{k+N}-R_{k+l-1+N})(R_{l+N}-R_{k+l-1+N})+R_{k+l-1+N}(1-R_{k+l-1+N}).
\end{align*}
Therefore
\begin{align*}
 & f_{k,l}(N-1)-f_{k,l}(N)\\
 & =(R_{k+l-1+N}-R_{k+N})f_{k,l-1}(N)\\
 & \quad+(R_{-k-l+1-N}-R_{-l-N})f_{k-1,l}(N)\\
 & \quad+\left((R_{k+N}-R_{k+l-1+N})(R_{l+N}-R_{k+l-1+N})+R_{k+l-1+N}(1-R_{k+l-1+N})\right)f_{k-1,l-1}(N+1).
\end{align*}
By the induction hypothesis, we have
\[
f_{k,l-1}(N),\,f_{k-1,l}(N)\in\mathbb{Q}(q)_{k+l-1},\qquad f_{k-1,l-1}(N+1)\in\mathbb{Q}(q)_{k+l-2}=\mathbb{Q}(q)_{k+l}.
\]
Furthermore, since $R_{m}-R_{m'}\in\mathbb{Q}(q)_{\mathrm{od}}$ and
$R_{m}(1-R_{m})\in\mathbb{Q}(q)_{\mathrm{ev}}$ for all $m,m'\in\mathbb{Z}\setminus\{0\}$,
we have
\begin{align*}
R_{k+l-1+N}-R_{k+N},R_{-k-l+1-N}-R_{-l-N} & \in\mathbb{Q}(q)_{\mathrm{od}},\\
(R_{k+N}-R_{k+l-1+N})(R_{l+N}-R_{k+l-1+N})+R_{k+l-1+N}(1-R_{k+l-1+N}) & \in\mathbb{Q}(q)_{\mathrm{ev}}.
\end{align*}
It follows that
\[
f_{k,l}(N)-f_{k,l}(N-1)\in\mathbb{Q}(q)_{k+l}.
\]
Consequently,
\[
f_{k,l}(N)=\sum_{N'=1}^{N}\left(f_{k,l}(N')-f_{k,l}(N'-1)\right)+f_{k,l}(0)\in\mathbb{Q}(q)_{k+l}.
\]
This completes the proof.
\end{proof}

\subsection{The case $A=0$, $B=C=\infty$, and $D=1$}

In this subsection, we prove the duality conjecture for the case $A=0$,
$B=C=\infty$, and $D=1$.

For $w\in\mathbb{Q}W^{0}$, define
\[
f(w):=\lim_{n\to\infty}L_{q}(w)_{A=q^{-n},B=C=\infty,D=1}\in\mathbb{Q}[[q]].
\]
Then we have
\[
f((u_{1}v_{1})\cdots(u_{k}v_{k}))=I_{q}(0;[u_{1}^{(1)},v_{1}^{(1)}],\dots,[u_{k}^{(k)},v_{k}^{(k)}];1)
\]
where
\[
A^{(j)}:=0,B^{(j)}:=C^{(j)}:=\infty
\]
and
\[
D^{(j)}:=q^{-\#\{h\geq j\,\mid\,u_{h}v_{h}\in\{AB,AC,BC\}\}}.
\]

\begin{lem}
\label{lem:case_no_AD}If $w\in\mathbb{Q}\langle AB,AC,BC,BD,CD\rangle\cap\mathbb{Q}W^{0}$,
then
\[
f(w)=f(\tau(w)).
\]
\end{lem}

\begin{proof}
It suffices to consider only the case $w\in\mathbb{Q}\langle AB,BD\rangle\cap\mathbb{Q}W^{0}$.
Let $w=e_{BD}^{l_{1}}e_{AB}^{k_{1}}\cdots e_{BD}^{l_{r}}e_{AB}^{k_{r}}$.
Then,
\[
f(w)=\zeta_{q}^{\mathrm{BZ}}(\{1\}^{l_{1}-1},k_{1}+1,\dots,\{1\}^{l_{r}-1},k_{r}+1)
\]
and
\[
f(\tau(w))=\zeta_{q}^{\mathrm{BZ}}(\{1\}^{k_{r}-1},l_{r}+1,\dots,\{1\}^{k_{1}-1},l_{1}+1).
\]
Thus, the claim follows from the duality for $\zeta_{q}^{\mathrm{BZ}}$
\cite[Corollary~3]{Bradley_qMZVs}.
\end{proof}
\begin{lem}
\label{lem:erase_AD}For $w=w_{1}e_{AD}w_{2}\in W^{0}$, we have
\[
f(w)=f(w_{1}e_{AB}w_{2})+f(w_{1}e_{BD}w_{2})+f(w_{1}w_{2}).
\]
\end{lem}

\begin{proof}
Let $w_{1}=e_{u_{1}v_{1}}\cdots e_{u_{l}v_{l}}$ and $e_{AD}w_{2}=e_{u_{l+1}v_{l+1}}\cdots e_{u_{k}v_{k}}$.
Then
\begin{align*}
f(w) & =I_{q}(0;[u_{1}^{(1)},v_{1}^{(1)}],\dots,[u_{l}^{(l)},v_{l}^{(l)}],[0,D^{(l+1)}],[u_{l+2}^{(l+2)},v_{l+2}^{(l+2)}]\dots,[u_{k}^{(k)},v_{k}^{(k)}];1)\\
 & =I_{q}(0;[u_{1}^{(1)},v_{1}^{(1)}],\dots,[u_{l}^{(l)},v_{l}^{(l)}],[0,\infty],[u_{l+2}^{(l+2)},v_{l+2}^{(l+2)}]\dots,[u_{k}^{(k)},v_{k}^{(k)}];1)\\
 & \qquad+I_{q}(0;[u_{1}^{(1)},v_{1}^{(1)}],\dots,[u_{l}^{(l)},v_{l}^{(l)}],[\infty,D^{(l+1)}],[u_{l+2}^{(l+2)},v_{l+2}^{(l+2)}]\dots,[u_{k}^{(k)},v_{k}^{(k)}];1)
\end{align*}
where
\[
A^{(j)}:=0,B^{(j)}:=C^{(j)}:=\infty
\]
and
\[
D^{(j)}:=q^{-\#\{h\geq j\,\mid\,u_{h}v_{h}\in\{AB,AC,BC\}\}}.
\]
The first term
\[
I_{q}(0;[u_{1}^{(1)},v_{1}^{(1)}],\dots,[u_{l}^{(l)},v_{l}^{(l)}],[0,\infty],[u_{l+2}^{(l+2)},v_{l+2}^{(l+2)}],\dots,[u_{k}^{(k)},v_{k}^{(k)}];1)
\]
is equal to
\begin{align*}
 & I_{q}(0;[u_{1}^{(1)}q^{-1},v_{1}^{(1)}q^{-1}],\dots,[u_{l}^{(l)}q^{-1},v_{l}^{(l)}q^{-1}],[0,\infty],[u_{l+2}^{(l+2)},v_{l+2}^{(l+2)}],\dots,[u_{k}^{(k)},v_{k}^{(k)}];1)\\
 & \quad+I_{q}(0;[u_{1}^{(1)}q^{-1},v_{1}^{(1)}q^{-1}],\dots,[u_{l}^{(l)}q^{-1},v_{l}^{(l)}q^{-1}],[u_{l+2}^{(l+2)},v_{l+2}^{(l+2)}],\dots,[u_{k}^{(k)},v_{k}^{(k)}];1)\\
 & =f(w_{1}e_{AB}w_{2})+f(w_{1}w_{2}),
\end{align*}
by Corollary~\ref{cor:difference_formula_0}. The second term
\[
I_{q}(0;[u_{1}^{(1)},v_{1}^{(1)}],\dots,[u_{l}^{(l)},v_{l}^{(l)}],[\infty,D^{(l+1)}],[u_{l+2}^{(l+2)},v_{l+2}^{(l+2)}]\dots,[u_{k}^{(k)},v_{k}^{(k)}];1)
\]
is equal to 
\[
f(w_{1}e_{BD}w_{2}).
\]
Thus
\[
f(w)=f(w_{1}e_{AB}w_{2})+f(w_{1}e_{BD}w_{2})+f(w_{1}w_{2}).
\]
\end{proof}
\begin{thm}
For $w\in W^{0}$ we have
\[
f(w)=f(\tau(w)).
\]
\end{thm}

\begin{proof}
It suffices to consider the case where $w$ is a monomial. Let $\kappa(w)$
denote the number of occurrences of $AD$ in $w$. We prove the claim
by induction on $\kappa(w)$. The case $\kappa(w)=0$ follows from
Lemma~\ref{lem:case_no_AD}. Assume that $\kappa(w)>0$. Then $w$
can be written as $w=w_{1}e_{AD}w_{2}$. Then
\begin{align*}
f(\tau(w)) & =f(\tau(w_{2})e_{AD}\tau(w_{1}))\\
 & =f(\tau(w_{2})e_{AB}\tau(w_{1}))+f(\tau(w_{2})e_{BD}\tau(w_{1}))+f(\tau(w_{2})\tau(w_{1}))\qquad(\text{by Lemma~\ref{lem:erase_AD}})\\
 & =f(w_{1}e_{CD}w_{2})+f(w_{1}e_{AC}w_{2})+f(w_{1}w_{2})\qquad(\text{by the induction hypothesis})\\
 & =f(w_{1}e_{BD}w_{2})+f(w_{1}e_{AB}w_{2})+f(w_{1}w_{2})\\
 & =f(w)\qquad(\text{by Lemma~\ref{lem:erase_AD}}),
\end{align*}
which completes the proof.
\end{proof}

\subsection{The case $A=D$}

In this subsection, we prove the duality conjecture for the case $A=D$.
\begin{thm}
If $A=D$, then Conjecture~\ref{conj:main} holds.
\end{thm}

\begin{proof}
Let
\[
w=u_{1}\cdots u_{k}
\]
and
\[
\tau(w)=u_{1}'\cdots u_{k}'.
\]
For $0\leq j\leq k$, put
\begin{align*}
a_{j} & :=\#\{h\le j\mid u_{h}\in\{BC,BD,CD\}\},\\
b_{j} & :=\#\{h\le j\mid u_{h}\notin\{AC,AD\}\}+\#\{h>j\mid u_{h}=CD\},\\
c_{j} & :=\#\{h\le j\mid u_{h}\notin\{AB,AD\}\}+\#\{h>j\mid u_{h}=BD\},\\
d_{j} & :=-\#\{h>j\mid u_{h}\in\{AB,AC,BC\}\},\\
a_{j}' & :=\#\{h<j\mid u_{h}'\in\{BC,BD,CD\}\},\\
b_{j}' & :=1+\#\{h<j\mid u_{h}'\notin\{AC,AD\}\}+\#\{h\geq j\mid u_{h}'=CD\},\\
c_{j}' & :=1+\#\{h<j\mid u_{h}'\notin\{AB,AD\}\}+\#\{h\geq j\mid u_{h}'=BD\},\\
d_{j}' & :=-\#\{h\geq j\mid u_{h}'\in\{AB,AC,BC\}\}.
\end{align*}
Then
\[
L_{q}(w)=\prod_{j=1}^{k}\omega_{j},\quad L_{q}(\tau(w))=\prod_{j=1}^{k}\omega_{j}'
\]
where
\[
\omega_{j}:=\begin{cases}
\frac{A}{A-Aq^{a_{j}}}-\frac{A}{A-Bq^{b_{j}}} & \text{if }u_{j}=e_{AB}\\
\frac{A}{A-Aq^{a_{j}}}-\frac{A}{A-Cq^{c_{j}}} & \text{if }u_{j}=e_{AC}\\
\frac{A}{A-Aq^{a_{j}}}-\frac{A}{A-Aq^{d_{j}}} & \text{if }u_{j}=e_{AD}\\
\frac{A}{A-Bq^{b_{j}}}-\frac{A}{A-Cq^{c_{j}}} & \text{if }u_{j}=e_{BC}\\
\frac{A}{A-Bq^{b_{j}}}-\frac{A}{A-Aq^{d_{j}}} & \text{if }u_{j}=e_{BD}\\
\frac{A}{A-Cq^{c_{j}}}-\frac{A}{A-Aq^{d_{j}}} & \text{if }u_{j}=e_{CD}
\end{cases}\qquad\omega_{j}':=\begin{cases}
\frac{A}{A-Aq^{a_{j}'}}-\frac{A}{A-Bq^{b_{j}'}} & \text{if }u_{j}'=e_{AB}\\
\frac{A}{A-Aq^{a_{j}'}}-\frac{A}{A-Cq^{c_{j}'}} & \text{if }u_{j}'=e_{AC}\\
\frac{A}{A-Aq^{a_{j}'}}-\frac{A}{A-Aq^{d_{j}'}} & \text{if }u_{j}'=e_{AD}\\
\frac{A}{A-Bq^{b_{j}'}}-\frac{A}{A-Cq^{c_{j}'}} & \text{if }u_{j}'=e_{BC}\\
\frac{A}{A-Bq^{b_{j}'}}-\frac{A}{A-Aq^{d_{j}'}} & \text{if }u_{j}'=e_{BD}\\
\frac{A}{A-Cq^{c_{j}'}}-\frac{A}{A-Aq^{d_{j}'}} & \text{if }u_{j}'=e_{CD}.
\end{cases}
\]
Then, since
\begin{align*}
a_{k+1-j}' & =\#\{h>j\mid u_{h}\in\{AB,AC,BC\}\}=-d_{j},\\
b_{k+1-j}' & =1+\#\{h>j\mid u_{h}\notin\{AD,BD\}\}+\#\{h\leq j\mid u_{h}=AB\}=1-a_{j}+b_{j}-d_{j},\\
c_{k+1-j}' & =1+\#\{h>j\mid u_{h}\notin\{AD,CD\}\}+\#\{h\leq j\mid u_{h}=AC\}=1-a_{j}+c_{j}-d_{j},\\
d_{k+1-j}' & =-\#\{h\leq j\mid u_{h}\in\{BC,BD,CD\}\}=-a_{j},
\end{align*}
we have
\[
\omega_{k+1-j}':=\begin{cases}
\frac{A}{A-Cq^{1-a_{j}+c_{j}-d_{j}}}-\frac{A}{A-Aq^{-a_{j}}} & \text{if }u_{j}=e_{AB}\\
\frac{A}{A-Bq^{1-a_{j}+b_{j}-d_{j}}}-\frac{A}{A-Aq^{-a_{j}}} & \text{if }u_{j}=e_{AC}\\
\frac{A}{A-Aq^{-d_{j}}}-\frac{A}{A-Aq^{-a_{j}}} & \text{if }u_{j}=e_{AD}\\
\frac{A}{A-Bq^{1-a_{j}+b_{j}-d_{j}}}-\frac{A}{A-Cq^{1-a_{j}+c_{j}-d_{j}}} & \text{if }u_{j}=e_{BC}\\
\frac{A}{A-Aq^{-d_{j}}}-\frac{A}{A-Cq^{1-a_{j}+c_{j}-d_{j}}} & \text{if }u_{j}=e_{BD}\\
\frac{A}{A-Aq^{-d_{j}}}-\frac{A}{A-Bq^{1-a_{j}+b_{j}-d_{j}}} & \text{if }u_{j}=e_{CD}.
\end{cases}
\]
Therefore, we have
\[
\frac{\omega_{j}}{\omega_{k+1-j}'}=\begin{cases}
q^{a_{j}}\frac{(A-Bq^{-a_{j}+b_{j}})(A-Cq^{1-a_{j}+c_{j}-d_{j}})}{(A-Bq^{b_{j}})(A-Cq^{1+c_{j}-d_{j}})} & \text{if }u_{j}=e_{AB}\\
q^{a_{j}}\frac{(A-Cq^{-a_{j}+c_{j}})(A-Bq^{1-a_{j}+b_{j}-d_{j}})}{(A-Cq^{c_{j}})(A-Bq^{1+b_{j}-d_{j}})} & \text{if }u_{j}=e_{AC}\\
1 & \text{if }u_{j}=e_{AD}\\
q^{a_{j}+d_{j}-1}\frac{(A-Bq^{1-a_{j}+b_{j}-d_{j}})(A-Cq^{1-a_{j}+c_{j}-d_{j}})}{(A-Bq^{b_{j}})(A-Cq^{c_{j}})} & \text{if }u_{j}=e_{BC}\\
q^{d_{j}}\frac{(A-Bq^{b_{j}-d_{j}})(A-Cq^{1-a_{j}+c_{j}-d_{j}})}{(A-Bq^{b_{j}})(A-Cq^{1-a_{j}+c_{j}})} & \text{if }u_{j}=e_{BD}\\
q^{d_{j}}\frac{(A-Cq^{c_{j}-d_{j}})(A-Bq^{1-a_{j}+b_{j}-d_{j}})}{(A-Cq^{c_{j}})(A-Bq^{1-a_{j}+b_{j}})} & \text{if }u_{j}=e_{CD}.
\end{cases}
\]
For $0\leq j\leq k$, put
\[
\Phi(j):=q^{a_{j}d_{j}}\frac{\prod_{l=1-a_{j}+b_{j}-d_{j}}^{b_{j}-d_{j}}(A-Bq^{l})\prod_{l=1-a_{j}+c_{j}-d_{j}}^{c_{j}-d_{j}}(A-Cq^{l})}{\prod_{l=1-a_{j}+b_{j}}^{b_{j}}(A-Bq^{l})\prod_{l=1-a_{j}+c_{j}}^{c_{j}}(A-Cq^{l})}.
\]
Then, since
\[
(a_{j},b_{j},c_{j},d_{j})=(a_{j-1},b_{j-1},c_{j-1},d_{j-1})+\begin{cases}
(0,1,0,1) & \text{if }u_{j}=e_{AB}\\
(0,0,1,1) & \text{if }u_{j}=e_{AC}\\
(0,0,0,0) & \text{if }u_{j}=e_{AD}\\
(1,1,1,1) & \text{if }u_{j}=e_{BC}\\
(1,1,0,0) & \text{if }u_{j}=e_{BD}\\
(1,0,1,0) & \text{if }u_{j}=e_{CD},
\end{cases}
\]
as a vector in $\mathbb{Z}^{4}$, we have
\[
\frac{\Phi(j)}{\Phi(j-1)}=\frac{\omega_{j}}{\omega_{k+1-j}'}.
\]
Thus,
\[
\frac{L_{q}(w)}{L_{q}(\tau(w))}=\frac{\prod_{j=1}^{k}\omega_{j}}{\prod_{j=1}^{k}\omega_{k+1-j}'}=\prod_{j=1}^{k}\frac{\Phi(j)}{\Phi(j-1)}=\frac{\Phi(k)}{\Phi(0)}=1.
\]
This proves the theorem.
\end{proof}

\end{document}